\documentclass{amsart}
\usepackage{amsfonts}

\usepackage{setspace}
\usepackage{graphicx}
\usepackage{eufrak}
\usepackage{mathrsfs}
\usepackage{yfonts}
\usepackage{hyperref}
\usepackage{graphicx}
\usepackage{amsmath, amsthm, amscd, amsfonts, amssymb, graphicx, color}
\usepackage{url}
\usepackage{enumerate}
\makeatletter
\let\reftagform@=\tagform@
\def\tagform@#1{\maketag@@@{(\ignorespaces\textcolor{blue}{#1}\unskip\@@italiccorr)}}
\renewcommand{\eqref}[1]{\textup{\reftagform@{\ref{#1}}}}
\makeatother
\usepackage{hyperref}
\hypersetup{colorlinks=true, linkcolor=red, anchorcolor=green,
    citecolor=cyan, urlcolor=red, filecolor=magenta, pdftoolbar=true}

\usepackage{epsfig}        
\usepackage{epic,eepic}       

\usepackage[margin=1.2in]{geometry}
\newtheorem{theorem}{Theorem}
\theoremstyle{plain}

\newtheorem{corollary}{Corollary}

\newtheorem{remark}{Remark}

\numberwithin{equation}{section}

\input{tcilatex}
\DeclareMathOperator{\spe}{sp}

\def\etal{et al.\,}
\begin{document}
    \title[Some Numerical radius inequalities]{Some Numerical radius inequalities}
    \author[M.W. Alomari]{Mohammad W. Alomari}
    \address{Department of Mathematics, Faculty of Science and
        Information Technology, Irbid National University, 2600 Irbid
        21110, Jordan.} \email{mwomath@gmail.com}

    \subjclass[2000]{Primary: 47A12, 47A30   Secondary: 15A60, 47A63.} \keywords{\v{C}eby\v{s}ev
        functional, Numerical radius, non-commutative operators}

\begin{abstract}
In this work, a pre-Gr\"{u}ss inequality for positive Hilbert
space operators is proved. So that,  some numerical radius
inequalities are proved. On the other hand, based on a
non-commutative  Binomial formula, a non-commutative  upper bound
for the numerical radius of the summand of two bounded linear
Hilbert space operators is proved. A commutative version is also
obtained as well.
\end{abstract}

    \maketitle

    \section{Introduction}

    Let $\mathscr{B}\left( \mathscr{H}\right) $ be the Banach algebra
    of all bounded linear operators defined on a complex Hilbert space
    $\left( \mathscr{H};\left\langle \cdot ,\cdot \right\rangle
    \right)$  with the identity operator  $1_\mathscr{H}$ in
    $\mathscr{B}\left( \mathscr{H}\right) $. A bounded linear
    operator $A$ defined on $\mathscr{H}$ is selfadjoint if and only
    if $ \left\langle {Ax,x} \right\rangle \in \mathbb{R}$ for all
    $x\in \mathscr{H}$. The spectrum of an operator $A$ is the set of all $\lambda \in \mathbb{C}$  for which the operator $\lambda I - A$ does not have a bounded linear operator inverse, and is denoted by
    $\spe\left(A\right)$. Consider the real vector space
    $\mathscr{B}\left( \mathscr{H}\right)_{sa}$ of self-adjoint
    operators on $ \mathscr{H}$ and its positive cone
    $\mathscr{B}\left( \mathscr{H}\right)^{+}$ of positive operators
    on $\mathscr{H}$. Also, $\mathscr{B}\left(
    \mathscr{H}\right)_{sa}^I$ denotes the convex set of bounded
    self-adjoint operators on the Hilbert space $\mathscr{H}$ with
    spectra in a real interval $I$. A
    partial order is naturally equipped on $\mathscr{B}\left(
    \mathscr{H}\right)_{sa}$ by defining $A\le B$ if and only if
    $B-A\in   \mathscr{B}\left( \mathscr{H}\right)^{+}$.  We write $A
    > 0$ to mean that $A$ is a strictly positive operator, or
    equivalently, $A \ge 0$ and $A$ is invertible. When $\mathscr{H} =
    \mathbb{C}^n$, we identify $\mathscr{B}\left( \mathscr{H}\right)$
    with the algebra $\mathfrak{M}_{n\times n}$ of $n$-by-$n$ complex
    matrices. Then, $\mathfrak{M}^{+}_{n\times n}$ is just the cone of
    $n$-by-$n$ positive semidefinite matrices.

    For a bounded linear operator $T$ on a Hilbert space
    $\mathscr{H}$, the numerical range $W\left(T\right)$ is the image
    of the unit sphere of $\mathscr{H}$ under the quadratic form $x\to
    \left\langle {Tx,x} \right\rangle$ associated with the operator.
    More precisely,
    \begin{align*}
    W\left( T \right) = \left\{ {\left\langle {Tx,x} \right\rangle :x
        \in \mathscr{H},\left\| x \right\| = 1} \right\}
    \end{align*}
    Also, the (maximum) numerical radius is defined by
    \begin{align*}
    w_{\max}\left( T \right) = \sup \left\{ {\left| \lambda\right|:\lambda
        \in W\left( T \right) } \right\} = \mathop {\sup }\limits_{\left\|
        x \right\| = 1} \left| {\left\langle {Tx,x} \right\rangle }
    \right|:=   w\left( T \right)
    \end{align*}
    and  the (minimum) numerical radius is defined to be
    \begin{align*}
    w_{\min}\left( T \right) = \inf \left\{ {\left| \lambda\right|:\lambda
        \in W\left( T \right) } \right\} = \mathop {\inf }\limits_{\left\|
        x \right\| = 1} \left| {\left\langle {Tx,x} \right\rangle }
    \right|.
    \end{align*}

    The spectral radius of an operator $T$ is defined to be
    \begin{align*}
    r\left( T \right) = \sup \left\{ {\left| \lambda\right|:\lambda
        \in \spe\left( T \right) } \right\}
    \end{align*}

    We recall that,  the usual operator norm of an operator $T$ is
    defined to be
    \begin{align*}
    \left\| T \right\| = \sup \left\{ {\left\| {Tx} \right\|:x \in
        H,\left\| x \right\| = 1} \right\}.
    \end{align*}
    and
    \begin{align*}
    \ell \left( T \right): &= \inf \left\{ {\left\| {Tx} \right\|:x
        \in \mathscr{H},\left\| x \right\| = 1} \right\}
    \\
    &=      \inf \left\{ {\left|\left\langle {Tx,y} \right\rangle
        \right|:x,y \in
        \mathscr{H},\left\| x \right\| =\left\| y \right\|= 1} \right\}.
    \end{align*}

it is well known that $w\left(\cdot\right)$ defines an operator
norm on $\mathscr{B}\left( \mathscr{H}\right) $ which is
equivalent to operator norm $\|\cdot\|$. Moreover, we have
\begin{align}
\frac{1}{2}\|T\|\le w\left(T\right) \le \|T\|\label{eq1.1}
\end{align}
for any $T\in \mathscr{B}\left( \mathscr{H}\right)$. The
inequality is sharp.

In 2003, Kittaneh \cite{FK1}  refined the right-hand side of
\eqref{eq1.1}, where he proved that
\begin{align}
w\left(T\right) \le
\frac{1}{2}\left(\|T\|+\|T^2\|^{1/2}\right)\label{eq1.2}
\end{align}
for any  $T\in \mathscr{B}\left( \mathscr{H}\right)$.

After that in 2005, the same author in \cite{FK} proved that
\begin{align}
\frac{1}{4}\|A^*A+AA^*\|\le  w^2\left(A\right) \le
\frac{1}{2}\|A^*A+AA^*\|.\label{eq1.3}
\end{align}
The inequality is sharp. This inequality was also reformulated and
generalized in \cite{EF} but in terms of Cartesian decomposition.

In 2007, Yamazaki \cite{Y} improved \eqref{eq1.1} by proving that
\begin{align}
w\left( T \right) \le \frac{1}{2}\left( {\left\| T \right\| +
w\left( {\widetilde{T}} \right)} \right) \le \frac{1}{2}\left(
{\left\| T \right\| + \left\| {T^2 } \right\|^{1/2} }
\right),\label{eq1.4}
\end{align}
where $\widetilde{T}=|T|^{1/2}U|T|^{1/2}$ with unitary $U$.

In 2008, Dragomir \cite{D4} (see also \cite{D1}) used Buzano
inequality to improve \eqref{eq1.1}, where he proved that
\begin{align}
w^2\left( T \right) \le \frac{1}{2}\left( {\left\| T \right\| +
w\left( {T^2} \right)} \right). \label{eq1.5}
\end{align}
This result was also recently generalized by Sattari \etal in
\cite{SMY}.

In \cite{SD1}, Dragomir studied the \v{C}eby\v{s}ev functional
    \begin{align*}
    \mathcal{C}\left( {f,g;A;x} \right) = \left\langle {f\left( A
        \right)g\left( A \right)x,x} \right\rangle  - \left\langle
    {f\left( A \right)x,x} \right\rangle \left\langle {g\left( A
        \right)x,x} \right\rangle
    \end{align*}
for any selfadjoint operator $A\in \mathcal{B}(H)$ and $x\in H$
with $\|x\|=1$. In particular, we have
\begin{align*}
\mathcal{C} \left( {f,f;A;x} \right) =   \left\langle {f^2 \left(
A \right)x,x} \right\rangle  - \left\langle {f\left( A \right)x,x}
\right\rangle ^2.
\end{align*}

In  the several works, Dragomir proved various bounds for the
\v{C}eby\v{s}ev functional. The most popular result  concerning
continuous synchronous (asynchronous) functions of selfadjoint
linear operators in Hilbert spaces, which reads
\begin{theorem}
    \label{thm1.1}Let $A\in \mathscr{B}\left( \mathscr{H}\right)_{sa}$   with
    $\spe\left(A\right)\subset \left[\gamma,\Gamma\right]$  for some
    real numbers $\gamma,\Gamma$ with $\gamma<\Gamma$. If $f,g: \left[
    {\gamma,\Gamma} \right]\to \mathbb{R}$ are continuous and
    synchronous (asynchronous) on $\left[ {\gamma,\Gamma} \right]$,
    then
    \begin{align*}
  \left\langle {f\left( A \right)g\left( A \right)x,x}
    \right\rangle \ge (\le) \left\langle { g\left( A \right)x,x}
    \right\rangle \left\langle { f\left( A \right)x,x} \right\rangle
    \end{align*}
    for any $x\in H$ with $\|x\|=1$.
\end{theorem}
This result was generalized recently by the author of this paper
in \cite{MA}.
For more related results concerning \v{C}eby\v{s}ev--Gr\"{u}ss type inequalities we refer the reader to \cite{SD2}, \cite{MM} and \cite{MB}.\\

\section{The Results}

    The following pre-Gr\"{u}ss inequality for linear bounded operators in inner product Hilbert spaces is valid.
    \begin{theorem}
        \label{t.4.1}Let $A\in \mathscr{B}\left( \mathscr{H}\right) ^+$.   If $f,g$  are both measurable functions on $\left[ 0,\infty\right)$,
        then we have the inequality
        \begin{align}
        \left| {\mathcal{C}\left( {f,g;A;x} \right)  } \right|
        \label{e.2.1}  \leq \mathcal{C}^{1/2}\left( {f,f;A;x} \right) \mathcal{C}^{1/2}\left( {g,g;A;x} \right)
        \end{align}%
        for any $x\in H$. In other words, we may write
        \begin{multline*}
        \left|\left\langle {f\left( A
            \right)g\left( A \right)x,x} \right\rangle  - \left\langle
        {f\left( A \right)x,x} \right\rangle \left\langle {g\left( A
            \right)x,x} \right\rangle \right|\\
        \le \left( {\left\langle {f^2 \left( A \right)x,x} \right\rangle  - \left\langle {f\left( A \right)x,x} \right\rangle ^2 } \right)^{1/2} \left( {\left\langle {g^2 \left( A \right)x,x} \right\rangle  - \left\langle {g\left( A\right)x,x} \right\rangle ^2 } \right)^{1/2}
        \end{multline*}
    \end{theorem}

    \begin{proof}
    It's not hard to show that
    \begin{align}
        \label{e.2.2}C\left( {f,g;A;x} \right)  = \frac{1}{2} \int_0^\infty  {\int_0^\infty  {\left( {f\left( t \right) - f\left( s \right)} \right)\left( {g\left( t \right) -
     g\left( s \right)} \right)d\left\langle {E_t x,x} \right\rangle d\left\langle {E_s x,x} \right\rangle } }
    \end{align}
            Utilizing the triangle inequality in \eqref{e.2.2} and then the Cauchy--Schwarz inequality, we get
        \begin{align*}
    \left| {C\left( {f,g;A;x} \right)} \right| &= \frac{1}{2}\left| {\int_0^\infty  {\int_0^\infty  {\left( {f\left( t \right) - f\left( s \right)} \right)\left( {g\left( t \right) - g\left( s \right)} \right)d\left\langle {E_t x,x} \right\rangle d\left\langle {E_s x,x} \right\rangle } } } \right| \\
    &\le \frac{1}{2}\int_0^\infty  {\int_0^\infty  {\left| {f\left( t \right) - f\left( s \right)} \right|\left| {g\left( t \right) - g\left( s \right)} \right|d\left\langle {E_t x,x} \right\rangle d\left\langle {E_s x,x} \right\rangle } }  \\
    &\le \frac{1}{2}\left( {\int_0^\infty  {\int_0^\infty  {\left| {f\left( t \right) - f\left( s \right)} \right|^2 d\left\langle {E_t x,x} \right\rangle d\left\langle {E_s x,x} \right\rangle } } } \right)^{1/2}  \\
    &\qquad \times \left( {\int_0^\infty  {\int_0^\infty  {\left| {g\left( t \right) - g\left( s \right)} \right|^2 d\left\langle {E_t x,x} \right\rangle d\left\langle {E_s x,x} \right\rangle } } } \right)^{1/2}  \\
    &= \frac{1}{2}\left( {\int_0^\infty  {d\left\langle {E_s x,x} \right\rangle } \int_0^\infty  {f^2 \left( t \right)d\left\langle {E_t x,x} \right\rangle }  - 2\int_0^\infty  {f\left( t \right)d\left\langle {E_t x,x} \right\rangle } \int_0^\infty  {f\left( s \right)d\left\langle {E_s x,x} \right\rangle } } \right. \\
    &\qquad\qquad\left. { + \int_0^\infty  {d\left\langle {E_t x,x} \right\rangle } \int_0^\infty  {f^2 \left( s \right)d\left\langle {E_s  x,x} \right\rangle } } \right)^{1/2}  \\
    &\qquad \times \left( {\int_0^\infty  {d\left\langle {E_s x,x} \right\rangle } \int_0^\infty  {g^2 \left( t \right)d\left\langle {E_t x,x} \right\rangle }  - 2\int_0^\infty  {g\left( t \right)d\left\langle {E_t x,x} \right\rangle } \int_0^\infty  {g\left( x \right)d\left\langle {E_s x,x} \right\rangle } } \right. \\
    &\qquad\qquad\left. { + \int_0^\infty  {d\left\langle {E_t x,x} \right\rangle } \int_0^\infty  {g^2 \left( s \right)d\left\langle {E_s x,x} \right\rangle } } \right)^{1/2}  \\
    &= \left( {1_{\mathscr{H}}  \cdot \int_0^\infty  {f^2 \left( t \right)d\left\langle {E_t x,x} \right\rangle }  - \left( {\int_0^\infty  {f\left( t \right)d\left\langle {E_t x,x} \right\rangle } } \right)^2 } \right)^{1/2}  \\
    &\qquad\times \left( {1_{\mathscr{H}}  \cdot \int_0^\infty  {g^2 \left( t \right)d\left\langle {E_t x,x} \right\rangle }  - \left( {\int_0^\infty  {g\left( t \right)d\left\langle {E_t x,x} \right\rangle } } \right)^2 } \right)^{1/2}  \\
    &= \left( {\left\langle {f^2 \left( A \right)x,x} \right\rangle  - \left\langle {f\left( A \right)x,x} \right\rangle ^2 } \right)^{1/2} \left( {\left\langle {g^2 \left( A \right)x,x} \right\rangle  - \left\langle {g\left( A\right)x,x} \right\rangle ^2 } \right)^{1/2}
    \end{align*}

        for any $x\in \mathscr{H}$, which gives
        the desired result (\ref{e.2.1}).
    \end{proof}
    \begin{corollary}
 Let $A\in \mathscr{B}\left( \mathscr{H}\right)  ^+$.   Then
    \begin{multline*}
    \left|\left\langle {Ax,x} \right\rangle  - \left\langle
    {A^{\alpha}x,x} \right\rangle \left\langle {A^{1-\alpha}x,x} \right\rangle \right|\\
    \le \left( {\left\langle {A^{2\alpha}x,x} \right\rangle  - \left\langle {A^{\alpha}x,x} \right\rangle ^2 } \right)^{1/2} \left( {\left\langle {A^{2\left(1-\alpha\right)}x,x} \right\rangle  - \left\langle {A^{1-\alpha}x,x} \right\rangle ^2 } \right)^{1/2}
    \end{multline*}
for any $x\in \mathscr{H}$ and all $\alpha\in
\left[0,\frac{1}{2}\right]$.

\end{corollary}

\begin{theorem}
    \label{t.4.3}Let $A\in \mathscr{B}\left( \mathscr{H}\right)  ^+$.   If $f,g$  are both measurable functions on $\left[ 0,\infty\right)$,
    then we have the inequality
\begin{multline}
w_{\max } \left( {f\left( A \right)g\left( A \right)} \right) -
w_{\min } \left( {f\left( A \right)} \right)\cdot w_{\min } \left(
{g\left( A \right)} \right)
\\
    \label{e.4.2}
 \le \left[ {\left\| {f\left( A \right)} \right\|^2  - \ell ^2 \left( {f^{1/2}\left( A \right)} \right)} \right]^{1/2} \cdot \left[ {\left\| {g\left( A \right)} \right\|^2  - \ell ^2 \left( {g^{1/2}\left( A \right)} \right)} \right]^{1/2}
\end{multline}
\end{theorem}
\begin{proof}
Using the basic triangle inequality
$\left|\left|a\right|-\left|b\right|\right|\le \left|a-b\right|$,
we have from \eqref{e.2.1} that
\begin{multline*}
\left|\left(\left| \left\langle {f\left( A
    \right)g\left( A \right)x,x} \right\rangle \right|\right) -\left(\left|  \left\langle
{f\left( A \right)x,x} \right\rangle \left\langle {g\left( A
    \right)x,x} \right\rangle \right|\right)\right|
\\
\le \left|\left\langle {f\left( A   \right)g\left( A \right)x,x}
\right\rangle  - \left\langle {f\left( A \right)x,x} \right\rangle
\left\langle {g\left( A\right)x,x} \right\rangle \right|
\\
\le\left( {\left\langle {f^2 \left( A \right)x,x} \right\rangle  -
\left\langle {f\left( A \right)x,x} \right\rangle ^2 }
\right)^{1/2} \left( {\left\langle {g^2 \left( A \right)x,x}
\right\rangle  - \left\langle {g\left( A\right)x,x} \right\rangle
^2 } \right)^{1/2}
\end{multline*}
Taking the supremum over $x  \in \mathscr{H}$, we obtain
\begin{align*}
&\mathop {\sup }\limits_{\left\| x \right\| = 1} \left| {\left|
{\left\langle {f\left( A \right)g\left( A \right)x,x}
\right\rangle } \right| - \left| {\left\langle {f\left( A
\right)x,x} \right\rangle } \right|\left| {\left\langle {g\left( A
\right)x,x} \right\rangle } \right|} \right|
\\
&\le \mathop {\sup }\limits_{\left\| x \right\| = 1}
\left|\left\langle {f\left( A  \right)g\left( A \right)x,x}
\right\rangle  - \left\langle {f\left( A \right)x,x} \right\rangle
\left\langle {g\left( A\right)x,x} \right\rangle \right|
\\
&\le \mathop {\sup }\limits_{\left\| x \right\| = 1} \left|
{\left\langle {f\left( A \right)g\left( A \right)x,x}
\right\rangle } \right| - \mathop {\inf }\limits_{\left\| x
\right\| = 1} \left\{ {\left| {\left\langle {f\left( A \right)x,x}
\right\rangle } \right|\left| {\left\langle {g\left( A \right)x,x}
\right\rangle } \right|} \right\}
\\
&\le \mathop {\sup }\limits_{\left\| x \right\| = 1} \left|
{\left\langle {f\left( A \right)g\left( A \right)x,x}
\right\rangle } \right| - \mathop {\inf }\limits_{\left\| x
\right\| = 1} \left| {\left\langle {f\left( A \right)x,x}
\right\rangle } \right| \cdot \mathop {\inf }\limits_{\left\| x
\right\| = 1} \left| {\left\langle {g\left( A \right)x,x}
\right\rangle } \right|
\\
&\le \mathop {\sup }\limits_{\left\| x \right\| = 1} \left[
{\left\| {f\left( A \right)x} \right\|^2  - \left\langle {f\left(
A \right)x,x} \right\rangle ^2 } \right]^{1/2} \cdot \mathop {\sup
}\limits_{\left\| x \right\| = 1} \left[ {\left\| {g\left( A
\right)x} \right\|^2  - \left\langle {g\left( A \right)x,x}
\right\rangle ^2 } \right]^{1/2}
\\
&\le \left[ {\mathop {\sup }\limits_{\left\| x \right\| = 1}
\left\| {f\left( A \right)x} \right\|^2  - \mathop {\inf
}\limits_{\left\| x \right\| = 1} \left\langle {f\left( A
\right)x,x} \right\rangle ^2 } \right]^{1/2} \cdot \left[ {\mathop
{\sup }\limits_{\left\| x \right\| = 1} \left\| {g\left( A
\right)x} \right\|^2  - \mathop {\inf }\limits_{\left\| x \right\|
= 1} \left\langle {g\left( A \right)x,x} \right\rangle ^2 }
\right]^{1/2}
\\
&= \left[ {\left\| {f\left( A \right)} \right\|^2  - \ell ^2
\left( {f^{1/2}\left( A \right)} \right)} \right]^{1/2} \cdot
\left[ {\left\| {g\left( A \right)} \right\|^2  - \ell ^2 \left(
{g^{1/2}\left( A \right)} \right)} \right]^{1/2}.
\end{align*}
It follows that
\begin{multline*}
w_{\max } \left( {f\left( A \right)g\left( A \right)} \right) -
w_{\min } \left( {f\left( A \right)} \right)w_{\min } \left(
{g\left( A \right)} \right)
\\
\le \left[ {\left\| {f\left( A \right)} \right\|^2  - \ell ^2
\left( {f^{1/2}\left( A \right)} \right)} \right]^{1/2} \cdot
\left[ {\left\| {g\left( A \right)} \right\|^2  - \ell ^2 \left(
{g^{1/2}\left( A \right)} \right)} \right]^{1/2},
\end{multline*}
or equivalently we have
\begin{multline*}
w_{\max } \left( {f\left( A \right)g\left( A \right)} \right) -
w_{\min } \left( {f\left( A \right)} \right)\cdot w_{\min } \left(
{g\left( A \right)} \right)
\\
 \le \left[ {\left\| {f\left( A \right)} \right\|^2  - \ell ^2 \left( {f^{1/2}\left( A \right)} \right)} \right]^{1/2} \cdot \left[ {\left\| {g\left( A \right)} \right\|^2  - \ell ^2 \left( {g^{1/2}\left( A \right)} \right)} \right]^{1/2},
\end{multline*}
which proves the desired result.
\end{proof}

 \begin{corollary}
    Let $A\in \mathscr{B}\left( \mathscr{H}\right) ^+$.  Then,
 \begin{align}
w_{\max } \left( A \right) - w_{\min } \left( {A^\alpha  }
\right)\cdot w_{\min } \left( {A^{1 - \alpha } } \right)
 \label{e.4.3}
 \le  \left[ {\left\| {A^\alpha  } \right\|^2  - \ell ^2 \left( {A^{\frac{\alpha}{2}}  } \right)} \right]^{1/2}  \cdot \left[ {\left\| {A^{1 - \alpha } } \right\|^2  - \ell ^2 \left( {A^{\frac{1-\alpha}{2}} } \right)} \right]^{1/2}
 \end{align}
    for each $x\in \mathscr{H}$. In particular, we have
  \begin{align}
\label{e.4.4} w_{\max } \left( A \right) - w^2_{\min } \left(
{A^{1/2}  } \right)
  \le \left\| {A^{1/2}  } \right\|^2  - \ell ^2 \left( {A^{1/4}  } \right)
 \end{align}
    for each $x\in \mathscr{H}$.
 \end{corollary}

\begin{corollary}
    \label{c.4.2}Let $A\in \mathscr{B}\left( \mathscr{H}\right)  ^+$.   If $f$  is measurable functions on $\left[ 0,\infty\right)$,
    then we have the inequality
    \begin{align}
    w_{\max } \left( {f^2\left( A \right)} \right) - w^2_{\min } \left( {f\left( A \right)} \right)  \label{e.4.5}
\le \left\| {f\left( A \right)} \right\|^2  - \ell ^2 \left(
{f^{1/2}\left( A \right)} \right)
    \end{align}
    for each $x\in \mathscr{H}$.
\end{corollary}
A generalization of \eqref{e.4.4} can be deduced from
\eqref{e.4.5} as follows:
\begin{corollary}
    \label{c.4.3}Let $A\in \mathscr{B}\left( \mathscr{H}\right) ^+$.  Then, for any $p>0$ the inequality
    \begin{align}
    w_{\max } \left( {A^{2p}} \right) - w^2_{\min } \left( {A^p} \right)    \label{e.4.6}
    \le \left\| {A^p} \right\|^2  - \ell ^2 \left( {A^{p/2}} \right)
    \end{align}
    holds for each $x\in \mathscr{H}$.
\end{corollary}

The Schwarz inequality for positive operators reads that if $A$ is
a positive operator in $\mathscr{B}\left(\mathscr{H}\right)$, then
\begin{align}
\left| {\left\langle {Ax,y} \right\rangle} \right|  ^2  \le
\left\langle {A x,x} \right\rangle \left\langle { A y,y}
\right\rangle, \qquad 0\le \alpha \le 1. \label{eq1.1a}
\end{align}
for any   vectors $x,y\in \mathscr{H}$.

In 1951, Reid \cite{R} proved an inequality which in some senses
considered a variant of Schwarz inequality. In fact, he proved
that for all operators $A\in \mathscr{B}\left( \mathscr{H}\right)
$ such that $A$ is positive and $AB$ is selfadjoint then
\begin{align}
\left| {\left\langle {ABx,y} \right\rangle} \right|  \le \|B\|
\left\langle {A x,x} \right\rangle, \label{eq1.2a}
\end{align}
for all $x\in \mathscr{H}$. In \cite{H}, Halmos presented his
stronger version of Reid inequality \eqref{eq1.2a} by replacing
$r\left(B\right)$ instead of $\|B\|$.

In 1952, Kato  \cite{TK} introduced a companion inequality of
\eqref{eq1.1a}, called  the mixed Schwarz inequality,  which
asserts
\begin{align}
\left| {\left\langle {Ax,y} \right\rangle} \right|  ^2  \le
\left\langle {\left| A \right|^{2\alpha } x,x} \right\rangle
\left\langle {\left| {A^* } \right|^{2\left( {1 - \alpha }
\right)} y,y} \right\rangle, \qquad 0\le \alpha \le 1.
\label{eq1.3a}
\end{align}
for all positive operators $A\in \mathscr{B}\left(
\mathscr{H}\right) $ and any vectors $x,y\in \mathscr{H}$, where
$\left|A\right|=\left(A^*A\right)^{1/2}$.

In 1988,  Kittaneh  \cite{FK4} proved  a very interesting
extension combining both the Halmos--Reid inequality
\eqref{eq1.2a} and the   mixed Schwarz inequality \eqref{eq1.3a}.
His result reads that
\begin{align}
\left| {\left\langle {ABx,y} \right\rangle } \right| \le
r\left(B\right)\left\| {f\left( {\left| A \right|} \right)x}
\right\|\left\| {g\left( {\left| {A^* } \right|} \right)y}
\right\|\label{kittaneh.ineq}
\end{align}
for any   vectors $x,y\in  \mathscr{H} $, where $A,B\in
\mathscr{B}\left( \mathscr{H}\right)$ such that $|A|B=B^*|A|$ and
$f,g$ are  nonnegative continuous functions  defined on
$\left[0,\infty\right)$ satisfying that $f(t)g(t) =t$ $(t\ge0)$.
Clearly, choose $f(t)=t^{\alpha}$ and $g(t)=t^{1-\alpha}$ with
$B=1_{\mathscr{H}}$ we refer to \eqref{eq1.3a}. Moreover, choosing
$\alpha=\frac{1}{2}$ some manipulations refer to Halmos version of
Reid inequality.

\begin{theorem}
    \label{t.4.4}Let $A\in \mathscr{B}\left( \mathscr{H}\right)$.   If $f,g$  are both positive continuous and $f(t)g(t)=t$ for all  $t\in \left[ 0,\infty\right)$,
    then we have the inequality
\begin{align}
w_{\max } \left( {A} \right) - w_{\min } \left( {f\left( A
\right)} \right)\cdot w_{\min } \left( {g\left( A \right)} \right)
\label{e.4.7} \le \frac{1}{2}\left\| {f^2 \left( {\left| A
\right|} \right) + g^2 \left( {\left| {A^* } \right|} \right)}
\right\|- \ell ^2 \left( {f^{1/2}\left( A \right)} \right)\cdot
\ell ^2 \left( {g^{1/2}\left( A \right)} \right).
\end{align}
\end{theorem}
\begin{proof}
Since  $f(t)g(t)=t$ for all  $t\in \left[ 0,\infty\right)$, then
from the proof of Theorem \ref{t.4.3} we have
\begin{align*}
&\mathop {\sup }\limits_{\left\| x \right\| = 1} \left| {\left|
{\left\langle {f\left( A \right)g\left( A \right)x,x}
\right\rangle } \right| - \left| {\left\langle {f\left( A
\right)x,x} \right\rangle } \right|\left| {\left\langle {g\left( A
\right)x,x} \right\rangle } \right|} \right|
\\
&\le \mathop {\sup }\limits_{\left\| x \right\| = 1} \left|
{\left\langle {f\left( A \right)g\left( A \right)x,x}
\right\rangle } \right| - \mathop {\inf }\limits_{\left\| x
\right\| = 1} \left\{ {\left| {\left\langle {f\left( A \right)x,x}
\right\rangle } \right|\left| {\left\langle {g\left( A \right)x,x}
\right\rangle } \right|} \right\}
\\
&= \mathop {\sup }\limits_{\left\| x \right\| = 1} \left|
{\left\langle {Ax,x} \right\rangle } \right| - \mathop {\inf
}\limits_{\left\| x \right\| = 1} \left| {\left\langle {f\left( A
\right)x,x} \right\rangle } \right| \cdot \mathop {\inf
}\limits_{\left\| x \right\| = 1} \left| {\left\langle {g\left( A
\right)x,x} \right\rangle } \right| \qquad\qquad (by \,\,
\text{\eqref{kittaneh.ineq}\,\, with \,\, $B=1_{\mathscr{H}}$})
\\
&\le \mathop {\sup }\limits_{\left\| x \right\| = 1} \left\langle
{f^2 \left( {\left| A \right|} \right)x,x} \right\rangle ^{1/2}
\left\langle {g^2 \left( {\left| {A^* } \right|} \right)x,x}
\right\rangle ^{1/2}
  - \mathop {\inf }\limits_{\left\| x \right\| = 1} \left| {\left\langle {f\left( A \right)x,x} \right\rangle } \right| \cdot \mathop {\inf }\limits_{\left\| x \right\| = 1} \left| {\left\langle {g\left( A \right)x,x} \right\rangle } \right|
\\
&\le \mathop {\sup }\limits_{\left\| x \right\| = 1} \left\langle
{f^2 \left( {\left| A \right|x,x} \right)} \right\rangle ^{1/2}
\left\langle {g^2 \left( {\left| {A^* } \right|x,x} \right)}
\right\rangle ^{1/2}
 - \mathop {\inf }\limits_{\left\| x \right\| = 1} \left| {\left\langle {f\left( A \right)x,x} \right\rangle } \right| \cdot \mathop {\inf }\limits_{\left\| x \right\| = 1} \left| {\left\langle {g\left( A \right)x,x} \right\rangle } \right|
 \\
&\le\frac{1}{2} \mathop {\sup }\limits_{\left\| x \right\| = 1}
\left\langle {\left[ {f^2 \left( {\left| A \right|} \right) + g^2
\left( {\left| {A^* } \right|} \right)} \right]x,x} \right\rangle
 - \mathop {\inf }\limits_{\left\| x \right\| = 1} \left| {\left\langle {f\left( A \right)x,x} \right\rangle } \right| \cdot \mathop {\inf }\limits_{\left\| x \right\| = 1} \left| {\left\langle {g\left( A \right)x,x} \right\rangle } \right|
 \end{align*}
 which proves the required result.
 \end{proof}

 \begin{corollary}
    \label{c.4.4}Let $A\in\mathscr{B}\left( \mathscr{H}\right) ^+ $.   If $f,g$  are both positive continuous and $f(t)g(t)=t$ for all  $t\in \left[ 0,\infty\right)$. Then
    \begin{align}
    w_{\max } \left( {A} \right) - w_{\min } \left( {A^{\alpha}} \right)\cdot w_{\min } \left( {A^{1-\alpha}} \right)
    \label{e.4.8}
    \le
    \frac{1}{2}\left\| {  \left| A \right|^{2\alpha}  +   \left| {A^* } \right|^{2(1-\alpha)}} \right\|- \ell ^2 \left( A^{\frac{\alpha}{2}}    \right)\cdot  \ell ^2 \left(
     A^{\frac{1-\alpha}{2}}  \right)
    \end{align}
 In particular, we have
    \begin{align}
w_{\max } \left( {A} \right) - w^2_{\min } \left( {A^{1/2}}
\right) \label{e.4.9} \le \frac{1}{2}\left\| {  \left| A \right|
+   \left| {A^* } \right|} \right\|- \ell ^4 \left( {A^{1/4}}
\right)
\end{align}
 \end{corollary}

 \begin{theorem}
Let $A,B\in \mathscr{B}\left( \mathscr{H}\right)$. Then,
 \begin{align}
\label{eq2.11} w\left( {\left( {A + B} \right)^2 } \right) \le
w\left( {A^2 } \right) + w\left( {B^2 } \right) + \frac{1}{4}
\min\left\{w\left( {BA^2 B} \right) + \left\| {AB} \right\|^2
,w\left( {AB^2 A} \right) + \left\| {BA} \right\|^2 \right\}
 \end{align}
 \end{theorem}

 \begin{proof}
 Let us first note that the Dragomir refinement  of  Cauchy-Schwarz
 inequality  reads that \cite{D5}:
 \begin{align*}
 \left| {\left\langle {x,y} \right\rangle } \right| \le \left|
 {\left\langle {x,e} \right\rangle \left\langle {e,y} \right\rangle
 } \right| + \left| {\left\langle {x,y} \right\rangle  -
    \left\langle {x,e} \right\rangle \left\langle {e,y} \right\rangle
 } \right| \le \left\| x \right\|\left\| y \right\|
 \end{align*}
 for all $x,y,e\in \mathscr{H}$ with $\|e\|=1$.

 It's easy to deduce the inequality
 \begin{align}
 \left| {\left\langle {x,e} \right\rangle \left\langle {e,y}
    \right\rangle } \right| \le \frac{1}{2}\left( {\left|
    {\left\langle {x,y} \right\rangle } \right| + \left\| x
    \right\|\left\| y \right\|} \right).\label{key}
 \end{align}
Utilizing the triangle inequality we have
 \begin{align}
\left| {\left\langle {\left(A+B\right)^2x,x} \right\rangle }
\right| \le \left| {\left\langle {A^2x,x} \right\rangle }
\right|+\left| {\left\langle {ABx,x} \right\rangle } \right|\left|
{\left\langle {x,A^*B^*x} \right\rangle } \right|+\left|
{\left\langle {B^2x,x} \right\rangle } \right|\label{eq2.13}
\end{align}
 so that by setting $e = u$, $x = ABu$, $y = A^*B^*u$ in \eqref{key} we get
  \begin{align*}
 \left| {\left\langle {ABu,u} \right\rangle \left\langle {u,A^*B^*u}
    \right\rangle } \right| \le \frac{1}{2}\left( {\left|
    {\left\langle {ABu,A^*B^*y} \right\rangle } \right| + \left\| ABu
    \right\|\left\| A^*B^*u \right\|} \right).
 \end{align*}
 Substituting in \eqref{eq2.13} and taking the supremum over all unit vector $x\in \mathscr{H}$ we get
\begin{align*}
w\left( {\left( {A + B} \right)^2 } \right) \le w\left( {A^2 }
\right) + w\left( {B^2 } \right) + \frac{1}{2}\left( {w\left(
{BA^2 B} \right) + \left\| {AB} \right\|^2 } \right).
\end{align*}
Replacing $B$ by $A$ and $A$ by $B$ in the previous inequality we
get that
\begin{align*}
w\left( {\left( {B + A} \right)^2 } \right) \le w\left( {B^2 }
\right) + w\left( {A^2 } \right) + \frac{1}{2}\left( {w\left(
{AB^2 A} \right) + \left\| {BA} \right\|^2 } \right).
\end{align*}
Adding the above two inequalities we get  the desired result.
 \end{proof}

\begin{corollary}
Let $A\in \mathscr{B}\left( \mathscr{H}\right)$. Then,
\begin{align}
\label{eq2.14}w\left( {A^2 } \right) \le \frac{1}{8}\left(
{w\left( {A^4 } \right) + \left\| {A^2 } \right\|^2 } \right)
\end{align}
  \end{corollary}
\begin{proof}
Setting $A=B$ in \eqref{eq2.11} we get the desired result.
\end{proof}

 Let $\mathscr{U}$ be an associative algebra, not necessarily commutative, with identity $1_{\mathscr{U}}$. For
 two elements $A$ and $B$ in $\mathscr{U}$, that commute; i.e., $AB=BA$. It's well known the Binomial Theorem reads that
\begin{align}
\label{eq2.15}\left( {A + B} \right)^n  = \sum\limits_{k = 0}^n
{\left( {\begin{array}{*{20}c}
        n  \\
        k  \\
        \end{array}} \right) A^kB^{n - k} }.
\end{align}

 In \cite{W}, Wyss derived an interesting non-commutative  Binomial formula for commutative algebra $\mathscr{U}$  with identity $1_{\mathscr{U}}$.  Denotes
 $\mathscr{L}\left(\mathscr{U}\right)$   the algebra of linear transformations from $\mathscr{U}$ to $\mathscr{U}$. Let $A,X\in \mathscr{U}$, the element (commutator) $d_A$
 in  $\mathscr{L}\left(\mathscr{U}\right)$ is defined by
 \begin{align*}
 d_A \left( X \right) = \left[ {A,X} \right] = AX - XA.
 \end{align*}
 It follows that, $A$ and $d_A$ are element of $\mathscr{L}\left(\mathscr{U}\right)$. Moreover, $A$ can be looked upon as an element in  $\mathscr{L}\left(\mathscr{U}\right)$  by $A\left(X\right)=AX$, which  is the left multiplication.

 The following properties are hold  \cite{W}:
\begin{enumerate}
\item  $A$ and $d_A$ commute; i.e.,
$Ad_A\left(X\right)=d_AA\left(X\right)$.

\item $d_A$ is a derivation on $\mathscr{U}$; i.e., $d_A \left(
{XY} \right) = \left( {d_A X} \right)Y + X\left( {d_A Y} \right)$
.

\item  $\left( {A - d_A } \right)X = XA$.

\item The Jacobi identity $d_A d_B \left( C \right) + d_B d_C
\left( A \right) + d_C d_A \left( B \right) = 0$  holds.

\end{enumerate}
Using these properties Wyss proved the following non-commutative
version of Binomial theorem  \cite{W}:
\begin{align}
\label{eq2.16}\left( {A + B} \right)^n  = \sum\limits_{k = 0}^n
{\left( {\begin{array}{*{20}c}
        n  \\
        k  \\
        \end{array}} \right)\left\{ {\left( {A + d_B } \right)^k 1_{\mathscr{U}} } \right\}B^{n - k} }
\end{align}
for all elements $A,B$ in the associative algebra $\mathscr{U}$
with identity $1_{\mathscr{U}}$.

 We write
\begin{align}
\label{eq2.17}\left( {A + d_B } \right)^n 1_{\mathscr{U}}  = A^n
+ D_n \left( {B,A} \right).
\end{align}
For a commutative algebra, $D_n(B,A)$ is identically zero. We thus
call $D_n(B,A)$ the essential non-commutative part. Moreover,
$D_n(B,A)$ satisfies the following recurrence relation
\begin{align*}
D_{n + 1} \left( {B,A} \right) = d_B A^n  + \left( {A + d_B }
\right)D_n \left( {B,A} \right), \qquad n\ge 0
\end{align*}
with $D_0\left(B,A\right)= 0$.

A  non-commutative  upper bound for the summand of two bounded
linear Hilbert space operators is proved in the following result.
\begin{theorem}
\label{thm6}Let $A, B\in \mathscr{B}\left( \mathscr{H}\right)$. If
$f,g$  are both positive continuous and $f(t)g(t)=t$ for all
$t\in \left[ 0,\infty\right)$. Then
 \begin{multline}
w\left( {\left( {A + B} \right)^n } \right)
\\
\le\frac{1}{2}\sum\limits_{k = 0}^n {\left( {\begin{array}{*{20}c}
        n  \\
        k  \\
        \end{array}} \right)\left\| {f\left( {\left| {\left\{ {\left( {A + d_B } \right)^k 1_{\mathscr{H}} } \right\}B^{n - k} } \right|} \right) + g\left( {\left| {\left( {B^{n - k} } \right)^* \left\{ {\left( {A + d_B } \right)^k 1_{\mathscr{H}} } \right\}^* } \right|} \right)} \right\|}
\label{eq2.18}
\end{multline}
    where $d_{B }\left( A \right) = \left[ {B,A} \right] = BA - AB$  and $d^*_{B }\left( A \right) = \left[ {B,A} \right]^* = A^*B^* - B^*A^*$.
\end{theorem}

\begin{proof}
By Utilizing the triangle inequality in \eqref{eq2.16} and by
employing \eqref{kittaneh.ineq} we have
 \begin{align*}
&\left| {\left\langle {\left(A+B\right)^nx,y} \right\rangle }
\right|
\\
&= \left| {\left\langle {\left( {\sum\limits_{k = 0}^n {\left(
{\begin{array}{*{20}c}
                    n  \\
                    k  \\
                    \end{array}} \right)\left\{ {\left( {A + d_B } \right)^k 1_{\mathscr{H}} } \right\}B^{n - k} } } \right)x,y} \right\rangle } \right|
\\
&\le\sum\limits_{k = 0}^n {\left( {\begin{array}{*{20}c}
        n  \\
        k  \\
        \end{array}} \right)\left| {\left\langle {\left( {\left\{ {\left( {A + d_B } \right)^k 1_{\mathscr{H}} } \right\}B^{n - k} } \right)x,y} \right\rangle } \right|}
\\
&\le \sum\limits_{k = 0}^n {\left( {\begin{array}{*{20}c}
        n  \\
        k  \\
        \end{array}} \right)\left\| {f\left( {\left| {\left\{ {\left( {A + d_B } \right)^k 1_{\mathscr{H}} } \right\}B^{n - k} } \right|} \right)x} \right\| \cdot \left\| {g\left( {\left| {\left( {B^{n - k} } \right)^* \left\{ {\left( {A + d_B } \right)^k 1_{\mathscr{H}} } \right\}^* } \right|} \right)y} \right\|}  \\
&\le \sum\limits_{k = 0}^n {\left( {\begin{array}{*{20}c}
        n  \\
        k  \\
        \end{array}} \right)\left\langle {f\left( {\left| {\left\{ {\left( {A + d_B } \right)^k 1_{\mathscr{H}} } \right\}B^{n - k} } \right|} \right)x,x} \right\rangle ^{1/2} \left\langle {g\left( {\left| {\left( {B^{n - k} } \right)^* \left\{ {\left( {A + d_B } \right)^k 1_{\mathscr{H}} } \right\}^* } \right|} \right)y,y} \right\rangle ^{1/2} }
\\
&\le \frac{1}{2}\sum\limits_{k = 0}^n {\left(
{\begin{array}{*{20}c}
        n  \\
        k  \\
        \end{array}} \right)\left[ {\left\langle {f\left( {\left| {\left\{ {\left( {A + d_B } \right)^k 1_{\mathscr{H}} } \right\}B^{n - k} } \right|} \right)x,x} \right\rangle  + \left\langle {g\left( {\left| {\left( {B^{n - k} } \right)^* \left\{ {\left( {A + d_B } \right)^k 1_{\mathscr{H}} } \right\}^* } \right|} \right)y,y} \right\rangle } \right]},
\end{align*}
where the last inequality follows by applying AM-GM inequality.
Hence, by letting $y=x$, we get
\begin{align*}
&\left| {\left\langle {\left(A+B\right)^nx,x} \right\rangle } \right|\\
&\le \frac{1}{2}\sum\limits_{k = 0}^n {\left(
{\begin{array}{*{20}c}
        n  \\
        k  \\
        \end{array}} \right)\left[ {\left\langle {f\left( {\left| {\left\{ {\left( {A + d_B } \right)^k 1_{\mathscr{H}} } \right\}B^{n - k} } \right|} \right)x,x} \right\rangle  + \left\langle {g\left( {\left| {\left( {B^{n - k} } \right)^* \left\{ {\left( {A + d_B } \right)^k 1_{\mathscr{H}} } \right\}^* } \right|} \right)x,x} \right\rangle } \right]}
\\
&\le \frac{1}{2}\sum\limits_{k = 0}^n {\left(
{\begin{array}{*{20}c}
        n  \\
        k  \\
        \end{array}} \right)\left\langle {\left\{ {f\left( {\left| {\left\{ {\left( {A + d_B } \right)^k 1_{\mathscr{H}} } \right\}B^{n - k} } \right|} \right) + g\left( {\left| {\left( {B^{n - k} } \right)^* \left\{ {\left( {A + d_B } \right)^k 1_{\mathscr{H}} } \right\}^* } \right|} \right)} \right\}x,x} \right\rangle }.
\end{align*}
Taking the supremum over all unit vector $x\in \mathscr{H}$ we get
the required result.
\end{proof}

\begin{remark}
\label{rem1}Taking the supremum over all unit vectors $x,y\in
\mathscr{H}$ in the proof of Theorem \ref{thm6} we get the
following power norm inequality:
\begin{align*}
\left\| {\left( {A + B} \right)^n } \right\|
\le\frac{1}{2}\sum\limits_{k = 0}^n {\left( {\begin{array}{*{20}c}
        n  \\
        k  \\
        \end{array}} \right)\left\| {f\left( {\left| {\left\{ {\left( {A + d_B } \right)^k 1_{\mathscr{H}} } \right\}B^{n - k} } \right|} \right) + g\left( {\left| {\left( {B^{n - k} } \right)^* \left\{ {\left( {A + d_B } \right)^k 1_{\mathscr{H}} } \right\}^* } \right|} \right)} \right\|}
\end{align*}
for all $A, B\in \mathscr{B}\left( \mathscr{H}\right)$.
\end{remark}

\begin{corollary}
\label{cor7}Let $A, B\in \mathscr{B}\left( \mathscr{H}\right)$. If
$f,g$  are both positive continuous and $f(t)g(t)=t$ for all
$t\in \left[ 0,\infty\right)$. Then
\begin{align}
\label{eq2.19}  w\left( { A + B   } \right)
    \le \frac{1}{2}\left\| {f\left( {\left| B \right|} \right) + g\left( {\left| {B^* } \right|} \right) + f\left( {\left| {A + d_B A} \right|} \right) + g\left( {\left| {\left( {A^*  + A^* d^*_{B } } \right)} \right|} \right)} \right\|
    \end{align}
    where $d_{B }\left( A \right) = \left[ {B,A} \right] = BA - AB$  and $d^*_{B }\left( A \right) = \left[ {B,A} \right]^* = A^*B^* - B^*A^*$.
\end{corollary}
\begin{proof}
Setting $n=1$ in \eqref{eq2.18} we get that
\begin{align*}
w\left( { A + B   } \right) \le\frac{1}{2}\left\| {f\left( {\left|
B \right|} \right) + g\left( {\left| {B^* } \right|} \right) +
f\left( {\left| {\left( {A + d_B } \right)1_{\mathscr{H}} }
\right|} \right) + f\left( {\left| {\left( {A + d_B } \right)^*
1_{\mathscr{H}} } \right|} \right)} \right\|.
\end{align*}
Making use of  \eqref{eq2.17}, we have
\begin{align*}
\left( {A + d_B } \right)1_{\mathscr{H}}  = A + D_1 \left( {B,A}
\right) = A + d_B A,
\end{align*}
and
\begin{align*}
\left( {A + d_B } \right)^* 1_{\mathscr{H}}  = \left( {A^*  +
d_B^* } \right)1_{\mathscr{H}} = A^*  + D_1 \left( {B^* ,A^* }
\right) = A^*  + A^* d_{B^* }.
\end{align*}
Hence,
\begin{align*}
w\left( { A + B   } \right) \le \frac{1}{2}\left\| {f\left(
{\left| B \right|} \right) + g\left( {\left| {B^* } \right|}
\right) + f\left( {\left| {A + d_B A} \right|} \right) + g\left(
{\left| {\left( {A^*  + A^* d^*_{B } } \right)} \right|} \right)}
\right\|
\end{align*}
which gives the required result.
\end{proof}
 \begin{remark}
 As noted in Remark \ref{rem1} and deduced in Corollary \ref{cor7}, we may observe that
 \begin{align*}
\left\| { A + B   } \right\|
 \le    \frac{1}{2}\left\| {f\left( {\left| B \right|} \right) + g\left( {\left| {B^* } \right|} \right) + f\left( {\left| {A + d_B A} \right|} \right) + g\left( {\left| {\left( {A^*  + A^* d^*_{B } } \right)} \right|} \right)} \right\|
 \end{align*}
 $A, B\in \mathscr{B}\left( \mathscr{H}\right)$.
 \end{remark}
 \begin{corollary}
\label{cor8}For $A, B\in\mathscr{B}\left( \mathscr{H}\right)$
that commute. If $f,g$  are both positive continuous and
$f(t)g(t)=t$ for all  $t\in \left[ 0,\infty\right)$. Then
    \begin{align}
    w\left( {\left( {A + B} \right)^n } \right) \le
    \frac{1}{2}\sum\limits_{k = 0}^n {\left( {\begin{array}{*{20}c}
            n  \\
            k  \\
            \end{array}} \right)\left\| {f\left( {\left| {A^k B^{n - k} } \right|} \right) + g\left( {\left| {\left( {B^{n - k} } \right)^* \left( {A^k } \right)^* } \right|} \right)} \right\|}.
    \label{eq2.20}
    \end{align}
In particular, we have
\begin{align*}
w\left( {A + B } \right)    \le \frac{1}{2}\left\| {f\left(
{\left| B \right|} \right) + g\left( {\left| {B^* } \right|}
\right) + f\left( {\left| A \right|} \right) + g\left( {\left|
{A^* } \right|} \right)} \right\|.
\end{align*}
\end{corollary}
\begin{proof}
Since $AB=BA$, then $d_B=0$ in \eqref{eq2.19}. Alternatively, we
may use \eqref{eq2.15} and proceed as in the proof of Theorem
\ref{thm6}.
\end{proof}
 \begin{remark}
As in the same way we previously remarked,  for $A, B\in
\mathscr{B}\left( \mathscr{H}\right)$  that commute, we can have
    \begin{align*}
    \left\| {\left( {A + B} \right)^n } \right\|    \le
    \frac{1}{2}\sum\limits_{k = 0}^n {\left( {\begin{array}{*{20}c}
            n  \\
            k  \\
            \end{array}} \right)\left\| {f\left( {\left| {A^k B^{n - k} } \right|} \right) + g\left( {\left| {\left( {B^{n - k} } \right)^* \left( {A^k } \right)^* } \right|} \right)} \right\|}.
    \end{align*}
    In particular,
    \begin{align*}
     \left\| {A + B } \right\|  \le
    \frac{1}{2}\left\| {f\left( {\left| B \right|} \right) + g\left( {\left| {B^* } \right|} \right) + f\left( {\left| A \right|} \right) + g\left( {\left| {A^* } \right|} \right)} \right\|.
    \end{align*}
Setting $f\left(t\right)=t^{\alpha}$ and
$g\left(t\right)=t^{1-\alpha}$ for all $\alpha\in
\left[0,1\right]$, in the last inequality above we get
\begin{align*}
\left\| {A + B } \right\|   \le \frac{1}{2}\left\| { \left| B
\right|^{\alpha} +  \left| {B^* } \right|^{1-\alpha} +  \left| A
\right|^{\alpha} + \left| {A^* } \right|^{1-\alpha}} \right\|.
\end{align*}
In special case for $\alpha=\frac{1}{2}$ we have,
\begin{align*}
\left\| {A + B } \right\|   \le \frac{1}{2}\left\| { \left| B
\right|^{1/2}  +  \left| {B^* } \right|^{1/2} +  \left| A
\right|^{1/2}  + \left| {A^* } \right|^{1/2} } \right\|.
\end{align*}
\end{remark}
\begin{corollary}
    For $A\in \mathscr{B}\left( \mathscr{H}\right)$. If $f,g$  are both positive continuous and $f(t)g(t)=t$ for all  $t\in \left[ 0,\infty\right)$. Then
    \begin{align}
    w\left( {A^n } \right)  \le
    \frac{1}{2} \left(\left\| {f\left( {\left| {A^n } \right|} \right) + g\left( {\left| {\left( {A^n } \right)^* } \right|} \right)} \right\| \right)
    \label{eq2.21}
    \end{align}
\end{corollary}
 \begin{proof}
    Setting $B=0$ in \eqref{eq2.18} we get the desired result.
In another way, one may set  $B=A$ in Corollary \ref{cor8}, so
that we get
\begin{align*}
    w\left( {A^n } \right)  \le
\frac{1}{2^{n+1}}\left\| {f\left( {\left| {A^n } \right|} \right)
+ g\left( {\left| {  \left( {A^n } \right)^* } \right|} \right)}
\right\|\cdot\sum\limits_{k = 0}^n {\left( {\begin{array}{*{20}c}
        n  \\
        k  \\
        \end{array}} \right)},
\end{align*}
but since $ \sum\limits_{k = 0}^n {\left( {\begin{array}{*{20}c}
        n  \\
        k  \\
        \end{array}} \right)}  = 2^n $, then we get the required result.
 \end{proof}

\begin{corollary}
Let  $A\in\mathscr{B}\left( \mathscr{H}\right)$. Then,
    \begin{align}
\label{eq2.22}  w\left( {A^n } \right)  \le
    \frac{1}{2} \left(\left\| { \left| {A^n } \right|^{\alpha}  +  \left| {\left( {A^n } \right)^* } \right|^{1-\alpha} } \right\| \right).
    \end{align}
In particular, we have
\begin{align}
\label{eq2.23}w\left( {A } \right)  \le \frac{1}{2} \left(\left\|
{ \left| {A} \right|^{\alpha}  +  \left| {A^* } \right|^{1-\alpha}
} \right\| \right).
\end{align}
\end{corollary}

\begin{proof}
    Setting $f\left(t\right)=t^{\alpha}$ and
$g\left(t\right)=t^{1-\alpha}$ in \eqref{eq2.21}.
\end{proof}

\begin{corollary}
Let $A\in \mathscr{B}\left( \mathscr{H}\right)$. Then,
\begin{align}
\label{eq2.24}w\left( {A } \right)  \le \frac{1}{2} \left(\left\|
{ \left| {A} \right|  +  1_{\mathscr{H}} } \right\| \right)\le
\frac{1}{4}\left( {1 + \left\| A \right\| + \sqrt {\left( {\left\|
A \right\| - 1} \right)^2  + 4\left\| A \right\|} } \right)
\end{align}
\end{corollary}

\begin{proof}
Letting $\alpha=1$ in \eqref{eq2.23}, we get the first inequality.
The second inequality follows by employing the norm estimates
\cite{FK3}:
\begin{align*}
\left\| {A+ B } \right\| \le \frac{1}{2}\left( {\left\| A \right\|
+ \left\| B \right\| + \sqrt
    {\left( {\left\| A \right\| - \left\| B \right\|} \right)^2  +
        4\left\| {A^{1/2} B^{1/2} } \right\|^2 } } \right),
\end{align*}
and then
\begin{align*}
\left\| {A^{1/2} B^{1/2} } \right\|  \le\left\| {A  B } \right\|
^{1/2}.
\end{align*}
 in the first inequality and use the fact that   $|\|A|\|=\|A\|$. In other words, we have
\begin{align*}
\left\| {\left| A \right| + 1_{\mathscr{H}} } \right\| &\le \frac{1}{2}\left(\left\| {\left| A \right|} \right\| + \left\| {\left| {1_{\mathscr{H}} } \right|} \right\| + \sqrt {\left( {\left\| {\left| {A^{1/2} } \right|} \right\| - 1} \right)^2  + 4\left\| {\left| A \right|^{1/2} 1_{\mathscr{H}} } \right\|^2 } \right) \\
&= \frac{1}{2}\left( {1 + \left\| A \right\| + \sqrt {\left(
{\left\| A \right\| - 1} \right)^2  + 4\left\| A \right\|} }
\right)
\end{align*}
which proves the required result.
\end{proof}


\begin{thebibliography}{9}
    \setlength{\itemsep}{5pt}

\bibitem{MA}  M.W. Alomari, Pompeiu--\v{C}eby\v{s}ev type
inequalities for selfadjoint operators in Hilbert spaces, {\em
Adv. Oper. Theory }, {\bf{3}} no. 3 (2018), 9--22.


\bibitem{SD1} S.S. Dragomir,  \v{C}eby\v{s}ev's type inequalities for functions
    of selfadjoint operators in Hilbert spaces, {\em  Linear and
Multilinear Algebra}, {\textbf 58} (7--8)  (2010), 805--814.

\bibitem{SD2}S.S. Dragomir,  Operator inequalities
    of the Jensen, \v{C}eby\v{s}ev and Gr\"{u}ss type, Springer, New
York,  2012.

\bibitem{D1}S.S. Dragomir, Inequalities for the Numerical Radius of Linear Operators in Hilbert Spaces, SpringerBriefs in Mathematics, 2013.


\bibitem{D4}S.S. Dragomir, Some inequalities for the norm and the numerical radius of linear operator in Hilbert spaces, {\em Tamkang J. Math.}, {\bf39 } (1) (2008), 1--7.

\bibitem{D5}S.S. Dragomir, Some refinements of schwarz inequality, {\em Simposional de Math Si Appl.}
Polytechnical Inst Timisoara, Romania, {\bf1–2} (1985), 13--16.

\bibitem{EF}M. El-Haddad and  F. Kittaneh,  Numerical radius inequalities for Hilbert space operators. II.
{\em Studia Math.}, {\bf182} (2) (2007), 133--140.

\bibitem{H} P.R. Halmos,  \textit{ A Hilbert space problem book}, Van Nostrand Company, Inc., Princeton,
N.J., 1967.


 \bibitem{TK}T. Kato,  Notes on some inequalities for linear operators, {\em Math. Ann.,} {\bf125} (1952),
208--212.


 \bibitem{FK}F. Kittaneh, Numerical radius inequalities for Hilbert space operators, {\em  Studia Math.}, {\bf 168} (1) (2005), 73--80.


\bibitem{FK1} F. Kittaneh, A numerical radius inequality and an estimate for the numerical radius of the Frobenius companion matrix, {\em Studia Math.} {\bf158} (2003), 11--17.


 \bibitem{FK3}F. Kittaneh, Norm inequalities for certian operator sums, {\em J. Funct. Anal.} {\bf143} (2) (1997), 337--348.


\bibitem{FK4}F. Kittaneh,  Notes on some inequalities for Hilbert Space operators, {\em Publ. Res. Inst. Math. Sci},
 {\bf24} (2) (1988), 283--293.


 \bibitem{MM}J.S. Matharu and M.S. Moslehian,  Gr\"{u}ss inequality for some
    types of positive linear maps, {\em J. Operator Theory}, {\textbf 73}
(1) (2015),  265--278.


\bibitem{MB}M.S. Moslehian and M. Bakherad,  Chebyshev type inequalities for
    Hilbert space operators, {\em  J. Math. Anal. Appl.}, {\textbf 420}
(2014), no. 1, 737--749.







\bibitem{R} W. Reid,  \textit{  Symmetrizable completely continuous linear tarnsformations in Hilbert
    space},   Duke Math.,  {\bf18} (1951), 41--56.

\bibitem{SMY}M. Sattari, M.S. Moslehian and T.  Yamazaki,  Some genaralized numerical radius inequalities
for Hilbert space operators, {\em Linear Algebra Appl}, {\bf470}
(2014), 1--12.

\bibitem{W}W. Wyss, Two non-commutative binomial theorems, Preprint (2017). \url{https://arxiv.org/abs/1707.03861}

\bibitem{Y} T. Yamazaki, On upper and lower bounds of the numerical radius and an equality condition, {\em Studia Math.}, {\bf178} (2007), 83--89.




\end{thebibliography}
    \end{document}